\documentclass[12pt]{amsart}
\usepackage[utf8]{inputenc}
\usepackage{mathrsfs} 
\usepackage{amssymb}
\usepackage{color}
\usepackage{bm}
\usepackage{graphicx}
\usepackage[centertags]{amsmath}
\usepackage{amsfonts}
\usepackage{amsthm}
\usepackage{enumerate}
\usepackage{tocvsec2}
\usepackage{xcolor}
\usepackage[margin=.75in]{geometry}

\newtheorem{theorem}{Theorem}[section]
\newtheorem*{theorem*}{Theorem}

\newtheorem{lemma}[theorem]{Lemma}
\newtheorem{rem}[theorem]{Remark}

\newtheorem{proposition}[theorem]{Proposition}

\theoremstyle{definition}

\newcommand{\nn}{\mathbb{N}}
\newcommand{\ee}{\varepsilon}

\newcommand{\ran}{\text{ran}}

\begin{document}

\title{$Sz(\cdot)\leqslant \omega^\xi$ is rarely a three space property}

\begin{abstract} We prove that for any non-zero, countable ordinal $\xi$ which is not additively indecomposable, the property  of having Szlenk index not exceeding $\omega^\xi$ is not a three space property.  This complements a result of Brooker and Lancien,  which states that if $\xi$ is additively indecomposable, then having Szlenk index not exceeding $\omega^\xi$ is a three space property.  

\end{abstract}

\author{R.M. Causey}
\address{Department of Mathematics, Miami University, Oxford, OH 45056, USA}
\email{causeyrm@miamioh.edu}


\maketitle

\section{Introduction}

Since its inception, the Szlenk index has proven to be a valuable tool with many applications in Banach space theory, including renorming (\cite{KOS}, \cite{GKL}, \cite{Cpower}, \cite{Cflat}), embedding and universality (\cite{Boss}, \cite{FOSZ}, \cite{C}), and factorization of operators (\cite{Brooker}, \cite{CN2}). See \cite{Lancien} for a survey of the Szlenk index and its applications. 

In this paper, we say that a property which a Banach space may or may not possess is a \emph{three space property} if whenever $X$ is a Banach space and $Y$ is a closed subspace of $X$ such that two of the spaces $X, Y, X/Y$ has the property, then the third space also has the property. It is well known that for a Banach space $X$ and a closed subspace $Y$ of $X$, $Sz(Y), Sz(X/Y)\leqslant Sz(X)$. Therefore in order to decide whether the property `$Sz(\cdot)\leqslant \omega^\xi$' is a three space property, it is necessary and sufficient to answer the question of whether $Sz(Y), Sz(X/Y)\leqslant \omega^\xi$ implies $Sz(X)\leqslant \omega^\xi$.  We recall that the Szlenk index of an Asplund Banach space is always $\omega^\xi$ for some ordinal $\xi$ (\cite{Lancien}), so that we lose no generality by considering only such ordinals.  In the case that $\xi=\omega^\gamma$ for some $\gamma$, Brooker and Lancien gave an affirmative answer.   

\begin{theorem}\cite{BL} Let $\gamma$ be any ordinal. Let $X$ be a Banach space and let $Y$ be a closed subspace. If two of the three spaces $X,Y,X/Y$ have Szlenk index not exceeding $\omega^{\omega^\gamma}$, then the third space also has Szlenk index not exceeding $\omega^{\omega^\gamma}$. 

\label{BLth}
\end{theorem}

We note that except for the case $\gamma=0$, this result was already shown by Lancien in \cite{Lancien1}. Theorem \ref{BLth} result was also shown in \cite{CAlt} by a method dual to that in \cite{BL}.     

We also recall that an ordinal $\xi$ is said to be a \emph{gamma number} if it is not the sum of two smaller ordinals.  It is a standard fact about ordinals that  $\xi$ is a gamma number if and only if $\xi$ is equal to $0$ or $\omega^\gamma$ for some $\gamma$. We recall that an ordinal is said to be \emph{additively indecomposable} if it is a non-zero gamma number. Since a Banach space $X$ has $Sz(X)\leqslant 1=\omega^0$ if and only if it is finite dimensional, combining this trivial fact with Theorem \ref{BLth} yields that `$Sz(\cdot)\leqslant \omega^\xi$' is a three space property whenever $\xi$ is a gamma number. To the best of our knowledge, prior to this writing, there were no ordinals $\xi$ for which it was known that `$Sz(\cdot)\leqslant \omega^\xi$' is not a three space property.  The main result of this paper is to show that for countable ordinals, the countable gamma numbers form a complete list of countable ordinals for which `$Sz(\cdot)\leqslant \omega^\xi$' is a three space property.

\begin{theorem} For $\xi\in (0,\omega_1)\setminus \{\omega^\gamma: \gamma<\omega_1\}$, the property of having Szlenk index not exceeding $\omega^\xi$ is not a three space property.

\label{main1}
\end{theorem}

Also, it was shown in \cite{BCM} that the property of not admitting an $\ell_1^\xi$-spreading model is a three space property whenever $\xi=\omega^\gamma$ for some countable $\gamma$.   As a matter of convention, this result is also true when $\xi=0$.  Again, prior to this writing, we are not aware of an ordinals $\xi$ for which it is known that not admitting an $\ell_1^\xi$-spreading model fails to be a three space property.  The spaces which we use to prove Theorem \ref{main1} will also give a complete characterization of the collection of $\xi$ such that not admitting an $\ell_1^\xi$-spreading model is a three space property.  

\begin{theorem} For $\xi\in (0, \omega_1)\setminus \{\omega^\gamma: \gamma<\omega_1\}$, the property of admitting no $\ell_1^\xi$-spreading model is not a three space property. 

\label{main2}
\end{theorem}

\section{Szlenk index and spreading models}

For a Banach space $X$, a weak$^*$-compact subset $K$ of $X^*$, and $\ee>0$, we let $s_\ee(K)$ denote the subset of $K$ consisting of those $x^*\in K$ such that for any weak$^*$-neighborhood $V$ of $x^*$, $\text{diam}(V\cap K)>\ee$.   We define the transfinite derivations as $$s^0_\ee(K)=K,$$ $$s^{\xi+1}(K)=s_\ee(s^\xi_\ee(K)),$$ and if $\xi$ is a limit ordinal, $$s^\xi_\ee(K)=\bigcap_{\zeta<\xi}s^\zeta_\ee(K).$$  If there exists an ordinal $\xi$ such that $s^\xi_\ee(K)=\varnothing$, then we let $Sz(K,\ee)$ denote the minimum such $\xi$. If no such $\xi$ exists, we agree to the convention that $Sz(K, \ee)=\infty$.   If for each $\ee>0$, there exists an ordinal $\xi$ such that $s^\xi_\ee(K)=\varnothing$, then we let $Sz(K)=\sup_{\ee>0}Sz(K,\ee)$. If $Sz(K, \ee)=\infty$ for some $\ee>0$, then we establish the convention $Sz(K)=\infty$. We also establish the convention that $\xi<\infty$ for any ordinal $\xi$. For a Banach space $X$, we define $Sz(X, \ee)=Sz(B_{X^*}, \ee)$ for each $\ee>0$ and $Sz(X)=Sz(B_{X^*})$. 

We recall that if $Y$ is a closed subspace of $X$, then $Sz(Y)\leqslant Sz(X)$. Furthermore, $Sz(\ell_1)=\infty$, since the set $K=\{\pm 1\}^\nn$ is a $2$-separated subset of $B_{\ell_1^*}$ with no weak$^*$-isolated points, from which it follows that $K\subset s^\xi_1(B_{\ell_1^*})$ for all $\xi$. Therefore if the Banach space $X$ contains an isomorphic copy $\ell_1$, then $Sz(X)=\infty$. 

Throughout, we identify subsets of $\nn$ with strictly increasing sequences of natural numbers in the usual way.  For an infinite subset $M$ of $\nn$, we let $[M]$ denote the set of infinite subsets of $M$. For non-empty sets $E,F$, we let $E<F$ denote the relation $\max E<\min F$. For a non-empty subset $E$ of $\nn$ and $n\in\nn$, we let $n\leqslant E$ denote the relation that $n\leqslant \min E$.  The relation $E<n$ is defined similarly.  If $E<F$, we let $E\smallfrown F$ denote the concatenation of $E$ with $F$. Given our identification of sets with increasing sequences, if $E<F$, $E\smallfrown F$ is identified with $E\cup F$. We let $E\prec F$ denote the relation that $E$ is a proper initial segment of $F$. That is, $E\prec F$ if either $E=\varnothing\neq F$ or $F=(n_i)_{i=1}^t$ and $E=(n_i)_{i=1}^s$ for some $1\leqslant s< t$.   We recall the \emph{Schreier families} $\mathcal{S}_\xi$, $\xi<\omega_1$. We let $$\mathcal{S}_0=\{\varnothing\}\cup \{(n): n\in\nn\},$$ $$\mathcal{S}_{\xi+1}=\{\varnothing\}\cup \Bigl\{\bigcup_{n=1}^t E_n: t\leqslant E_1<\ldots <E_t, \varnothing\neq E_n\in\mathcal{S}_\xi\Bigr\},$$ and if $\xi<\omega_1$ is a limit ordinal, we fix $\xi_n\uparrow \xi$ and let $$\mathcal{S}_\xi=\{\varnothing\}\cup \{E: \exists n\leqslant E\in\mathcal{S}_{\xi_n}\}.$$ We let $MAX(\mathcal{S}_\xi)$ denote the set of $E\in \mathcal{S}_\xi$ which are maximal in $\mathcal{S}_\xi$ with respect to inclusion.

 We also recall that for $0\leqslant \alpha, \beta<\omega_1$, $$\mathcal{S}_\beta[\mathcal{S}_\alpha]:=\{\varnothing\}\cup \Bigl\{\bigcup_{n=1}^t E_n: E_1<\ldots <E_t, \varnothing \neq E_n\in \mathcal{S}_\alpha, (\min E_n)_{n=1}^t\in \mathcal{S}_\beta\Bigr\}.$$

We recall that for two sets $(m_i)_{i=1}^l$, $(n_i)_{i=1}^l$, $(n_i)_{i=1}^l$ is said to be a \emph{spread} of $(m_i)_{i=1}^l$ if $m_i\leqslant n_i$ for all $1\leqslant i\leqslant l$.  We now collect the necessary facts concerning these families.

\begin{proposition}\cite[Proposition $3.2$]{OTW}\cite[Propositions $3.1$,$3.2$]{CConcerning} \begin{enumerate}[(i)]\item For each $\xi<\omega_1$ and $E\in \mathcal{S}_\xi$, either $E\in MAX(\mathcal{S}_\xi)$ or $E\smallfrown (n)\in \mathcal{S}_\xi$ for all $E<n$. \item $\mathcal{S}_\xi$ contains no infinite $\preceq$-ascending chains. \item $\mathcal{S}_\xi$ contains all subsets of its members. \item $\mathcal{S}_\xi$ is \emph{spreading}. That is, $\mathcal{S}_\xi$ contains all spreads of its members.  \item For any $\alpha, \beta<\omega_1$, there exists $(m_n)_{n=1}^\infty\in[\nn]$ such that for any $E\in \mathcal{S}_{\alpha+\beta}$, $(m_n)_{n\in E}\in \mathcal{S}_\beta[\mathcal{S}_\alpha]$. \end{enumerate}

\label{trivlist}

\end{proposition}

Throughout, $c_{00}$ will denote the space of eventually zero scalar sequences and $(e_i)_{i=1}^\infty$ will denote the canonical basis of $c_{00}$. Given $x=\sum_{i=1}^\infty a_ie_i\in c_{00}$, we let $\ran(x)$ be the smallest interval in $\nn$ which contains $\{i\in \nn: a_i\neq 0\}$. Given $E\subset \nn$ and $x=\sum_{i=1}^\infty a_ie_i\in c_{00}$, we let $Ex=\sum_{i\in E}a_ie_i$.    For $0<\xi<\omega_1$, a sequence $(x_n)_{n=1}^\infty$ in some Banach space $X$ is an $\ell_1^\xi$-\emph{spreading model} if it is bounded and $$0<\inf\Bigl\{\|x\|: E\in\mathcal{S}_\xi, x=\sum_{n\in E}a_nx_n, \sum_{n\in E}|a_n|=1\Bigr\}.$$  For completeness, we say $(x_n)_{n=1}^\infty$ is an $\ell_1^0$-\emph{spreading model} if it is a seminormalized basic sequence.  We note here that by the almost monotone property of the  Schreier families (see \cite{OTW}), if $X$ admits an $\ell_1^\xi$-spreading model, then it admits an $\ell_1^\zeta$-spreading model for each $\zeta<\xi$. 

\begin{rem}\upshape Let $\xi<\omega_1$ be an ordinal. \begin{enumerate}[(i)]\item If the Banach space $X$ admits a weakly null $\ell_1^\xi$-spreading model, then $Sz(X)>\omega^\xi$.  Indeed, if $(x_n)_{n=1}^\infty \subset B_X$ is weakly null and $$0<\inf\Bigl\{\|x\|: E\in\mathcal{S}_\xi, x=\sum_{n\in E}a_nx_n, \sum_{n\in E}|a_n|=1\Bigr\},$$ then by \cite[Proposition $5$]{OSZ} combined with the main theorem of \cite{AJO} , the collection $(x_E)_{E\in \mathcal{S}_\xi\setminus \{\varnothing\}}$ given by $x_E=x_{\max E}$ witnesses that $Sz(X)>\omega^\xi$.

\item The Banach space $X$ admits an $\ell_1^0$-spreading model if and only if it is infinite dimensional if and only if $Sz(X)>1=\omega^0$. If $0<\xi$ and if the Banach space $X$ admits an $\ell_1^\xi$-spreading model, then either it contains an isomorhpic copy of $\ell_1$, in which case  $Sz(X)=\infty>\omega^\xi$, or $X$ admits a weakly null $\ell_1^\xi$-spreading model (see \cite[Remark $5$]{BCM}), and $Sz(X)>\omega^\xi$ by $(i)$.  Therefore any Banach space which admits an $\ell_1^\xi$-spreading model (weakly null or otherwise) must have Szlenk index exceeding $\omega^\xi$. \item For $\alpha, \beta<\omega_1$,  if a Banach space $X$ admits a bounded sequence $(x_n)_{n=1}^\infty$ such that $$0<\inf\Bigl\{\|\sum_{i\in E}a_ix_i\|_Y: E\in \mathcal{S}_\beta[\mathcal{S}_\alpha], \sum_{i\in E}|a_i|=1\Bigr\},$$ then the sequence $(x_n)_{n=1}^\infty$ admits a subsequence which is an $\ell_1^{\alpha+\beta}$-spreading model, and by $(ii)$,  $Sz(X)>\omega^{\alpha+\beta}$.  Indeed, the indicated subsequence can be taken to be $(x_{m_n})_{n=1}^\infty$, where $(m_n)_{n=1}^\infty\in [\nn]$ is as in Proposition \ref{trivlist}$(v)$.   \end{enumerate}

\label{remark}
\end{rem}

For $\gamma<\omega_1$, the \emph{Schreier space of order} $\gamma$, denoted by $X_\gamma$, is the completion of $c_{00}$ with respect to the norm $$\|x\|_{X_\gamma}=\sup \{\|Ex\|_{\ell_1}: E\in \mathcal{S}_\gamma\}.$$ The \emph{Baernstein} $2$ \emph{space of order} $\gamma$, denoted by $X_\gamma^2$,  is the completion of $c_{00}$ with respect to the norm $$\|x\|_{X^2_\gamma}= \sup\Bigl\{\bigl(\sum_{n=1}^\infty \|E_nx\|_{X_\gamma}^2\bigr)^{1/2}: E_1<E_2<\ldots\Bigr\}.$$  We note that $X_0^2=\ell_2$ isometrically.

We recall the relevant properties of the Schreier and Baernstein spaces. 

\begin{proposition} Fix $\gamma<\omega_1$, \begin{enumerate}[(i)] \item The canonical $c_{00}$ basis is normalized and $1$-unconditional in $X_\gamma$ and $X_\gamma^2$. \item $Sz(X_\gamma)=\omega^{\gamma+1}$ and $Sz(X_\gamma^2)=\omega^{\gamma+1}$. In particular, neither $X_\gamma$ nor $X_\gamma^2$ admits an $\ell_1^{\gamma+1}$-spreading model.\item The canonical $c_{00}$ basis is shrinking in both $X_\gamma$ and $X_\gamma^2$. \item The canonical $c_{00}$ basis of $X_\gamma^2$ satisfies a $1$-lower $\ell_2$ estimate. That is, for any $n\in\nn$ and any integers $0=p_0<\ldots <p_n$ and scalars $(a_i)_{i=1}^{p_n}$, $$\Bigl\|\sum_{i=1}^{p_n} a_ie_i\Bigl\|_{X_\gamma^2}^2 \geqslant \sum_{j=1}^n \Bigl\|\sum_{i=p_{j-1}+1}^{p_j} a_ie_i\Bigl\|_{X_\gamma^2}^2.$$ \item  For any $t\in\nn$ and integers $1\leqslant p_1<\ldots<p_{t+1}$ and $x_1, \ldots, x_t$ such that for each $1\leqslant n\leqslant t$,  $x_n\in B_{X_\gamma^2}\cap \text{\emph{span}}\{e_i: p_n \leqslant i< p_{n+1}\}$, and for any scalars  $(a_n)_{n=1}^t$, $$\Bigl\|\sum_{n=1}^t a_nx_n\Bigr\|_{X_\gamma^2} \leqslant 4\Bigl\|\sum_{n=1}^t a_ne_{p_n}\Bigr\|_{X_\gamma^2}.$$ \item The canonical basis of $X_\gamma^2$ is $1$-\emph{right dominant}. That is, for any sequences of positive integers $l_1<l_2<\ldots$, $m_1<m_2<\ldots$ such that $l_n\leqslant m_n$ for all $n\in\nn$, and for any $(a_n)_{n=1}^\infty\in c_{00}$, $$\Bigl\|\sum_{n=1}^\infty a_ne_{l_n}\Bigr\|_{X_\gamma^2} \leqslant \Bigl\|\sum_{n=1}^\infty a_ne_{m_n}\Bigr\|_{X_\gamma^2}.$$    \end{enumerate}

\label{important}
\end{proposition}

\begin{proof} Item $(i)$ is clear. The part of  $(ii)$ concerning the Szlenk index can be found in \cite[Proposition $4.6$]{C} and \cite[Proposition $4.2$]{CF}. The part concerning spreading models can be deduced by combining the part of $(ii)$ concerning Szlenk index with Remark \ref{remark}$(ii)$.   Item $(iii)$ can be deduced from the fact that since $Sz(X_\gamma), Sz(X_\gamma^2)<\infty$, neither of these spaces  can contain an isomorphic copy of $\ell_1$. Since the canonical $c_{00}$ basis is unconditional in each of these and neither space contains an isomorphic copy of $\ell_1$, the canonical $c_{00}$ basis is shrinking in $X_\gamma$ and in $X_\gamma^2$.    

For item $(iv)$, we first note that for any interval $I\subset \nn$,  $x\in \text{span}\{e_i:i\in I\}$, and $E_1<E_2<\ldots$, if $J=\{n\in\nn: E_n\cap I\neq 0\}$, then $$\sum_{n=1}^\infty \|E_nx\|_{X_\gamma}^2 = \sum_{n\in J} \|(I\cap E_n) x\|_{X_\gamma}^2.$$  Therefore $$\|x\|_{X^2_\gamma}^2 = \sup\Bigl\{\Bigl(\sum_{n=1}^m \|E_nx\|^2_{X_\gamma}\Bigr)^{1/2}:E_1<\ldots <E_m, E_n\subset \ran(x)\Bigr\}.$$   Now if $n$, $p_0, \ldots, p_n$, $(a_i)_{i=1}^{p_n}$ are as in the proposition, then for each $1\leqslant j\leqslant n$, we can choose   $E^j_1<\ldots <E^j_{k_j}$, $E^j_i\subset (p_{j-1}, p_j]$, such that $$\Bigl\|\sum_{i=p_{j-1}+1}^{p_j} a_ie_i\Bigr\|^2_{X^2_\gamma} = \sum_{l=1}^{k_n} \Bigl\|E_l\sum_{i=p_{j-1}+1}^{p_j} a_ie_i\Bigr\|_{X_\gamma}^2.$$   Enumerate $E^1_1, \ldots, E^1_{k_1}, E^2_1, \ldots, E^n_{k_n}$ as $E_1, \ldots, E_m$ and note that  $$\Bigl\|\sum_{i=1}^{p_n} a_ie_i\Bigr\|_{X_\gamma^2}^2 \geqslant \sum_{l=1}^m \Bigl\|E_l\sum_{j=1}^n\sum_{i=p_{j-1}+1}^{p_j} a_ie_i\Bigr\|_{X_\gamma}^2 = \sum_{j=1}^n  \Bigl\|\sum_{i=p_{j-1}+1}^{p_j} a_ie_i\Bigr\|_{X_\gamma}^2.$$  

Items $(v)$ and $(vi)$ were shown in \cite[Lemma $3.2$]{CF}.

\end{proof}

We recall that a (finite or infinite) sequence $(g_n)_{n=1}^N\subset c_{00}$ of non-zero vectors is said to be a \emph{block sequence with respect to the canonical} $c_{00}$ \emph{basis} provided there exist integers $0=p_0<\ldots <p_N$ (resp. $0=p_0<p_1<\ldots$ if $N=\infty$) such that for each $n\leqslant N$ (resp. $n<N$ if $N=\infty$) such that $g_n\in \text{span}\{e_i: p_{n-1}<i\leqslant p_n\}$. 

We last recall that for a Banach space $G$ for which the canonical $c_{00}$ basis is a Schauder basis, $0<\xi<\omega_1$, and a collection $(g_E)_{E\in \mathcal{S}_\xi\setminus \{\varnothing\}}\subset G$, $(g_E)_{E\in \mathcal{S}_\xi\setminus \{\varnothing\}}$ is said to be a \begin{enumerate}[(i)]\item \emph{weakly null tree} if for any $E\in \mathcal{S}_\xi\setminus MAX(\mathcal{S}_\xi)$, $(g_{E\cup (n)})_{n> \max E}$ is a weakly null sequence in $G$,  \item \emph{block tree} provided that for each $\varnothing \prec E_1\prec E_2\prec \ldots \prec E_n\in \mathcal{S}_\xi$, $(g_{E_i})_{i=1}^n$ is a block sequence with respect to the canonical $c_{00}$ basis, and for each $E\in \mathcal{S}_\xi\setminus MAX(\mathcal{S}_\xi)$, $(g_{E\cup (n)})_{n>\max E}$ is a block sequence with respect to the canonical $c_{00}$ basis. \end{enumerate}

We note that if $(g_E)_{E\in \mathcal{S}_\xi\setminus \{\varnothing\}}$ is a block tree, then for any $\varnothing\neq E\in\mathcal{S}_\xi$, $\ran(g_E)\geqslant  \max E.$

\section{The example}

Our example is a modification, and in some cases coincides with, spaces constructed by Lindenstrauss \cite{L}. Fix two countable ordinals $\alpha, \beta<\omega_1$.  Fix a dense sequence $(f_i)_{i=1}^\infty$ in the unit sphere of the Schreier space $X_\alpha$. Define a norm $\|\cdot\|_{G_{\alpha,\beta}}$ on $c_{00}$ by $$\Bigl\|\sum_{n=1}^\infty a_ne_n\Bigr\|_{G_{\alpha,\beta}} = \sup\Bigl\{   \Bigl\|\sum_{n=1}^\infty\|\sum_{i\in I_n} a_if_i\|_{X_\alpha}e_{\min I_n}\Bigr\|_{X_\beta^2}: I_1<I_2<\ldots, I_n\text{\ an interval}\Bigr\}.$$    Let $G_{\alpha,\beta}$ be the completion of $c_{00}$ with respect to this norm. In the case $\beta=0$, since $X_0^2=\ell_2$,  this construction coincides with that of Lindenstrauss in \cite{L}.  

\begin{rem}\upshape
Note that by  $1$-right dominance of the $X_\beta^2$ basis, we can compute the $\|\cdot\|_{G_{\alpha,\beta}}$ norm of some non-zero $g\in c_{00}$ as $$\|g\|_{G_{\alpha,\beta}}=\sup\Bigl\{   \Bigl\|\sum_{n=1}^m\|\sum_{i\in I_n} a_if_i\|_{X_\alpha}e_{\min I_n}\Bigr\|_{X_\beta^2}: I_1<\ldots <I_m, I_n\text{\ an interval}, I_n\subset \ran(g)\Bigr\}.$$    Indeed, for any intervals $J_1<J_2<\ldots$, if $K_1<\ldots <K_m$ is a list of those $J_i$ such that $J_i\cap \ran(g)\neq \varnothing$ and if $I_n= K_n\cap \ran(g)$ for $n=1, \ldots, m$, then by $1$-right dominance of the canonical $X_\beta^2$ basis together with the fact that $\sum_{i\in K_n}a_if_i=\sum_{i\in I_n}a_if_i$ and $\min I_n\geqslant \min K_n$ for each $1\leqslant n\leqslant m$,  $$  \Bigl\|\sum_{n=1}^\infty\|\sum_{i\in J_n} a_if_i\|_{X_\alpha}e_{\min J_n}\Bigr\|_{X_\beta^2} =\Bigl\|\sum_{n=1}^m\|\sum_{i\in K_n} a_if_i\|_{X_\alpha}e_{\min K_n}\Bigr\|_{X_\beta^2} \leqslant \Bigl\|\sum_{n=1}^m\|\sum_{i\in I_n} a_if_i\|_{X_\alpha}e_{\min I_n}\Bigr\|_{X_\beta^2} .$$  

\label{easyrem}
\end{rem}

We collect the following easy properties of $G_{\alpha,\beta}$. 

\begin{proposition} \begin{enumerate}[(i)]\item The canonical $c_{00}$ basis is a boundedly-complete, normalized, bimonotone basis for $G_{\alpha,\beta}$.\item If $g=\sum_{i=1}^\infty a_ie_i\in G_{\alpha,\beta}$, then $Qg:=\sum_{i=1}^\infty a_if_i$ is convergent in $X_\alpha$, and $Q:G_{\alpha,\beta}\to X_\alpha$ is  a quotient map. \end{enumerate}

\label{tech1}
\end{proposition}

\begin{proof}$(i)$ It is obvious that $\|e_i\|_{G_{\alpha, \beta}}=1$ for each $i\in\nn$.  Fix $(a_i)_{i=1}^\infty\in c_{00}$ and an interval $J$.  Let $b_i=a_i$ for $i\in J$ and $b_i=0$ for $i\in\nn\setminus J$.   As noted in Remark \ref{easyrem}, there exist intervals $I_1<\ldots <I_k$, $I_i\subset J$, such that $$\Bigl\|\sum_{i=1}^\infty b_ie_i\Bigr\|_{G_{\alpha, \beta}}  =\Bigl\|\sum_{n=1}^k \|\sum_{i\in I_n} b_if_i\|_{X_\alpha} e_{\min J_n} \Bigr\|_{X_\beta^2}=\Bigl\|\sum_{n=1}^k \|\sum_{i\in I_n} a_if_i\|_{X_\alpha} e_{\min J_n} \Bigr\|_{X_\beta^2} \leqslant \Bigl\|\sum_{i=1}^\infty a_ie_i\Bigr\|_{G_{\alpha, \beta}}.  $$ Therefore $(e_i)_{i=1}^\infty$ is bimonotone in $G_{\alpha, \beta}$.

We next show that $(e_i)_{i=1}^\infty$ is boundedly-complete. Suppose $(a_i)_{i=1}^\infty$, $\ee>0$, and $p_0<p_1<\ldots$ are such that $$\inf_n \Bigl\|\sum_{i=p_{n-1}+1}^{p_n} a_ie_i\Bigr\|_{G_{\alpha, \beta}}\geqslant\ee.$$  Then by Remark \ref{easyrem}, for each $n\in\nn$,  there exist intervals $I^n_1<\ldots <I^n_{k_n}$ such that $I^n_i\subset (p_{n-1}, p_n]$ and $$\ee\leqslant \Bigl\|\sum_{j=1}^{k_n} \|\sum_{i\in I^n_j} a_if_i\|_{X_\alpha} e_{\min I^n_j}\Bigr\|_{X_\beta^2}.$$  Then with $u_n=\sum_{j=1}^{k_n} \|\sum_{i\in I^n_j} a_if_i\|_{X_\alpha} e_{\min I^n_j}$, it follows that  $\ran(u_n)\subset (p_{n-1}, p_n]$ and $\|u_n\|_{X_\beta^2}\geqslant \ee$.   Let $I_1<I_2<\ldots $ be an enumeration of $I^1_1, I^1_2, \ldots, I^1_{k_1}, I^2_1, I^2_2 ,\ldots$. Using the $1$-lower $\ell_2$ estimate of the canonical $X^2_\beta$ basis, for any $m\in\nn$, \begin{align*} \Bigl\|\sum_{i=1}^{p_m} a_ie_i\Bigr\|_{G_{\alpha, \beta}} & \geqslant \Bigl\|\sum_{n=1}^\infty \|\sum_{i\in I_n} a_if_i\|_{X_\alpha} e_{\min I_i}\Bigr\|_{X^2_\beta} = \Bigl\|\sum_{n=1}^m\sum_{j=1}^{k_n} \|\sum_{i\in I^n_j} a_if_i\|_{X_\alpha}\Bigr\|_{X^2_\beta} = \Bigl\|\sum_{n=1}^m u_n\Bigr\|_{X^2_\beta} \geqslant \ee m^{1/2}. \end{align*}   By contraposition, if $\sup_m \|\sum_{i=1}^m a_ie_i\|_{G_{\alpha, \beta}}<\infty$, then $\sum_{i=1}^\infty a_ie_i$ is convergent in $G_{\alpha, \beta}$.

$(ii)$ Note that for any scalars $(a_i)_{i=1}^\infty\in c_{00}$ and any finite interval $I$, $$\Bigl\|\sum_{i=1}^\infty a_ie_i\Bigr\|_{G_{\alpha, \beta}}\geqslant \Bigl\|\|\sum_{i\in I}a_if_i\|_{X_\alpha}e_{\min I}\Bigr\|_{X^2_\beta}= \Bigl\|\sum_{i\in I} a_if_i\Bigr\|_{X_\alpha}.$$   Fix $g=\sum_{i=1}^\infty a_ie_i\in G_{\alpha, \beta}$.   If there exist $\ee>0$ and $p_0<p_1<\ldots$ such that $\inf_n \|\sum_{i=p_{n-1}+1}^{p_n}a_if_i\|_{X_\alpha} \geqslant \ee$, then the inequalities $$\inf_n \Bigl\|\sum_{i=p_{n-1}+1}^{p_n} a_ie_i\Bigr\|_{G_{\alpha, \beta}}\geqslant \Bigl\|\sum_{i=p_{n-1}+1}^{p_n}a_if_i\Bigr\|_{X_\alpha}\geqslant \ee$$ and $$\sup_m \Bigl\|\sum_{n=1}^{p_m} a_if_i\Bigr\|_{X_\alpha} \leqslant \sup_m \Bigl\|\sum_{i=1}^{p_m} a_ie_i\Bigr\|_{G_{\alpha, \beta}}=\|g\|_{G_{\alpha, \beta}}<\infty$$ contradict $(i)$.   Therefore $\sum_{i=1}^\infty a_if_i$ is convergent in $X_\alpha$, and $Q\sum_{i=1}^\infty a_ie_i=\sum_{i=1}^\infty a_if_i$ is well-defined.   By the first line of $(ii)$, $\|Q\|\leqslant 1$. Of course, $\|Q\|\geqslant 1$ by considering its action on any basis vector.   Furthermore, since $QB_{G_{\alpha, \beta}}$ contains $(f_n)_{n=1}^\infty$, which is dense in $S_{X_\alpha}$, $Q$ is a quotient map.

\end{proof}

Let $Q$ be as defined in Proposition \ref{tech1}$(ii)$ and let $H_{\alpha,\beta}=\ker(Q)$.  We recall the notation that for an interval $I\subset \nn$, $I$ also denotes the projection on $c_{00}$ given by $I\sum_{i=1}^\infty a_ie_i=\sum_{i\in I}a_ie_i$. Then for an interval $I\subset \nn$, $QI\sum_{i=1}^\infty a_ie_i=\sum_{i\in I}a_if_i$, which allows for a more convenient expression of the $\|\cdot\|_{G_{\alpha, \beta}}$ norm: $$\|g\|_{G_{\alpha, \beta}}= \sup\Bigl\{\Bigl\|\sum_{n=1}^\infty \|QI_ng\|_{X_\alpha}e_{\min I_n}\Bigr\|_{X_\beta^2}:I_1 <I_2<\ldots, I_n\subset\nn\text{\ an interval}\Bigr\}.$$

\begin{lemma} \begin{enumerate}[(i)]\item If $(g_E)_{E\in \mathcal{S}_{\beta+1}\setminus \{\varnothing\}}\subset B_{G_{\alpha, \beta}}$ is a block tree such that $\|Qg_E\|_{X_\alpha}\leqslant \ee$ for all $\varnothing \neq E\in \mathcal{S}_{\beta+1}$, then $\inf\{\|g\|_{G_{\alpha, \beta}}: E\in \mathcal{S}_{\beta+1}, g\in \text{\emph{co}}(g_F: \varnothing\prec F\preceq E)\} \leqslant \ee.$  \item $Sz(H_{\alpha, \beta})\leqslant \omega^{\beta+1}$.  \item $Sz(H_{\alpha, \beta})= \omega^{\beta+1}$ and $H_{\alpha, \beta}$ contains an $\ell_1^\beta$-spreading model. \end{enumerate}

\label{finish}

\end{lemma}

In the proof, we will use the \emph{repeated averages hierarchy} from \cite{AMT}, but with the notation from \cite{CN}. A complete presentation of the repeated averages hierarchy is unnecessary for our purposes. For readability, we recall here only the properties necessary for the proof below. Proofs of these properties can be found in \cite{AMT}. For each ordinal $\xi<\omega_1$ and each infinite subset $N$ of $\nn$, $\mathbb{S}^\xi_{N,1}:\nn\to [0,1]$ is a non-negative function such that $\sum_{i=1}^\infty \mathbb{S}^\xi_{N,1}(i)=1$ and, if $(n_i)_{i=1}^t$ is the maximal initial segment of $N$ which lies in $\mathcal{S}_\xi$, $(n_i)_{i=1}^t=\{i\in\nn: \mathbb{S}^\xi_{N,1}(i)\neq 0\}$.

\begin{proof}[Proof of Lemma \ref{finish}]$(i)$ We will repeatedly use facts from Proposition \ref{important}. Fix a block tree $(g_E)_{E\in\mathcal{S}_{\beta+1}\setminus \{\varnothing\}}\subset B_{G_{\alpha, \beta}}$ such that $\|Qg_E\|_{X_\alpha}\leqslant \ee$ for all $\varnothing\neq E\in\mathcal{S}_{\beta+1}$.   Combining the fact that  $Sz(X^2_\beta)=\omega^{\beta+1}$ with Remark \ref{remark}, the canonical basis of $X_\beta^2$ has no subsequence which is an $\ell_1^{\beta+1}$ spreading model.  Since the canonical basis of $X_\beta^2$ is shrinking and normalized, it is weakly null. Since the basis is weakly null and has no subsequence which is an $\ell_1^{\beta+1}$-spreading model, for any $\delta>0$, it follows from \cite[Theorem $4.12$]{CN} that there exists $L\in[\nn]$ such that for any $N\in[L]$, $$\Bigl\|\sum_{i=1}^\infty \mathbb{S}^{\beta+1}_{N, 1}(i)e_i\Bigr\|_{X_\beta^2}<\delta.$$  Let $n_1=\min L$ and recursively choose $n_1<n_2<\ldots$, $n_i\in L$, such that if  $1<t\in\nn$ is such that $(n_1, \ldots, n_t)\in \mathcal{S}_{\beta+1}$, then $\max \ran(g_{(n_1, \ldots, n_{t-1})})<n_t$.   Since $\mathcal{S}_{\beta+1}$ contains no infinite, $\preceq$-ascending chains, there exists $t\in \nn$ such that $E:=(n_i)_{i=1}^t$ is the maximal initial segment of $N$ which lies in $\mathcal{S}_{\beta+1}$. By the properties of the repeated averages hierarchy, $E=\{i: \mathbb{S}^{\beta+1}_{N,1}(i)\neq 0\}$ and $\sum_{i=1}^t \mathbb{S}^{\beta+1}_{N,1}(n_i)=1$.     Let $$g=\sum_{i=1}^t \mathbb{S}^{\beta+1}_{N,1}(n_i)g_{(n_1, \ldots, n_i)}\in \text{co}(g_F: \varnothing\prec F\preceq E).$$  Our goal is to estimate $\|g\|_{G_{\alpha, \beta}}$. For convenience, for $i=1, \ldots, t$,  let $g_i=g_{(n_1, \ldots, n_i)}$  and $w_i=\mathbb{S}^{\beta+1}_{N,1}(n_i)$. Then $g=\sum_{i=1}^t w_ig_i$. 

We fix intervals $I_1<I_2<\ldots$ such that $$\|g\|_{G_{\alpha, \beta}} = \Bigl\|\sum_{n=1}^\infty \|QI_ng\|_{X_\alpha}e_{\min I_n}\Bigr\|_{X^2_\beta}.$$  By including additional $I_n$ if necessary and using $1$-unconditionality of the canonical $X_\beta^2$ basis, we may assume that $\nn=\cup_{n=1}^\infty I_n$.    For each $n\in\nn$, let $A_n=\{i\in \{1, \ldots, t\}: \ran(g_i)\subset I_n\}$ and $A=\cup_{n=1}^\infty A_n$, $B=\{1, \ldots, t\}\setminus A$.   Then using the pairwise disjointness of $A_1, A_2, \ldots$ together with the fact that $\|Qg_i\|_{X_\alpha}\leqslant \ee$ for each $1\leqslant i\leqslant t$ and $I_ng_i=g_i$ for all $n\in\nn$ and $i\in A_n$, \begin{align*} \|g\|_{G_{\alpha, \beta}} & \leqslant \Bigl\|\sum_{n=1}^\infty \|QI_n \sum_{i\in A}w_ig_i\|_{X_\alpha} e_{\min I_n}\Bigr\|_{X^2_\beta} + \Bigl\|\sum_{n=1}^\infty \|QI_n \sum_{i\in B}w_ig_i\|_{X_\alpha} e_{\min I_n}\Bigr\|_{X^2_\beta} \\ & \leqslant \sum_{n=1}^\infty \sum_{i\in A_n} w_i\|QI_ng_i\|_{X_\alpha}+ \Bigl\|\sum_{n=1}^\infty \|QI_n \sum_{i\in B}w_ig_i\|_{X_\alpha} e_{\min I_n}\Bigr\|_{X^2_\beta} \\ & = \sum_{n=1}^\infty \sum_{i\in A_n} w_i\|Qg_i\|_{X_\alpha}+ \Bigl\|\sum_{n=1}^\infty \|QI_n \sum_{i\in B}w_ig_i\|_{X_\alpha} e_{\min I_n}\Bigr\|_{X^2_\beta} \\ &\leqslant \ee+ \Bigl\|\sum_{n=1}^\infty \|QI_n \sum_{i\in B}w_ig_i\|_{X_\alpha} e_{\min I_n}\Bigr\|_{X^2_\beta}. \end{align*} If $B=\varnothing$, then the second sum here is zero and we are done.  Assume $B\neq \varnothing$. Since $\nn=\cup_{n=1}^\infty I_n$, it follows that for each $i\in B$, there exist at least two values of $n\in\nn$ such that $I_ng_i\neq 0$. Enumerate $B=(b_1, \ldots, b_s)$ and let $B_1=(b_i: i\leqslant s, i\text{\ odd})$ and $B_2=(b_i: i\leqslant s, i\text{\ even})$.    For $i\in B$, let  $$C_i=\{n\in\nn: I_n g_i\neq 0\}.$$  Note that the sets $(C_i)_{i\in B_1}$ are pairwise disjoint, as are the sets $(C_i)_{i\in B_2}$. To see this, note that if $b_i, b_j\in B_1$ with $i<j$ and $n\in C_{b_i}\cap C_{b_j}$, then since $i,j$ are both odd, $b_i<b_{i+1}<b_j$, and  $$\ran(g_{b_{i+1}})\subset (\max \ran(g_{b_i}), \min \ran(g_{b_j}))\subset I_n.$$ But this means $b_{i+1}\in A_n$,  contradicting the fact that $b_{i+1}\in B=\{1, \ldots, t\}\setminus \cup_{m=1}^\infty A_m$. A similar argument yields that $(C_i)_{i\in B_2}$  are pairwise disjoint.  

We now turn to estimating $$\Bigl\|\sum_{n=1}^\infty \|QI_n \sum_{i\in B_1} w_ig_i\|_{X_\alpha}e_{\min I_n}\Bigr\|_{X^2_\beta} = \Bigl\|\sum_{i\in B_1} w_i \sum_{n\in C_i}\|QI_ng_i\|_{X_\alpha}e_{\min I_n}\Bigr\|_{X^2_\beta}.$$   For each $i\in B_1$ and $n\in C_i$, let $J_n=\ran(g_i)\cap I_n\neq \varnothing$.  Let $h_i=\sum_{n\in C_i} \|QJ_n g_i\|_{X_\alpha} e_{\min J_n}$ and $h'_i=\sum_{n\in C_i}\|QI_n g_i\|_{X_\alpha}e_{\min I_n}$. It follows from this definition that $\|h_i\|_{X_\beta^2}\leqslant \|g_i\|_{G_{\alpha, \beta}}\leqslant 1$. Moreover,  $\ran(h_i)\subset \ran(g_i)$. For each $i\in B_1$, $\|QJ_ng_i\|_{X_\alpha}=\|QI_ng_i\|_{X_\alpha}$ and $\min J_n\geqslant \min I_n$ for each $n\in C_i$. Also, by our choice of $n_1, \ldots, n_t$,  \begin{align*} n_1 & \leqslant \min \ran(g_1)=\min \ran(g_{(n_1)}) < n_2 \leqslant \min \ran(g_2)=\min \ran(g_{(n_1, n_2)}) <\ldots \\ & <n_t \leqslant \min \ran(g_t)=\min \ran(g_{(n_1, \ldots, n_t)}).\end{align*}  Therefore by the definition of $h_i'$, $1$-right dominance, Proposition \ref{important}$(v)$, the properties of $\mathbb{S}^{\beta+1}_{N,1}$, and our choice of $N$, \begin{align*}\Bigl\|\sum_{i\in B_1} w_i \sum_{n\in C_i}\|QI_ng_i\|_{X_\alpha}e_{\min I_n}\Bigr\|_{X^2_\beta} & = \Bigl\|\sum_{i\in B_1} w_i h'_i\Bigr\|_{X_\beta^2} \leqslant \Bigl\|\sum_{i\in B_1} w_ih_i\Bigr\|_{X^2_\beta} \leqslant 4\Bigl\|\sum_{i=1}^t w_ie_{n_i} \Bigr\|_{X^2_\beta} \\ & = 4\Bigl\|\sum_{i=1}^t \mathbb{S}^{\beta+1}_{N,1}(n_i)e_{n_i}\Bigr\|_{X^2_\beta} = 4\Bigl\|\sum_{i=1}^\infty \mathbb{S}^{\beta+1}_{N,1}(i)e_i\Bigr\|_{X^2_\beta} \\ & <4\delta. \end{align*}

An identical argument yields that \begin{align*} \Bigl\|\sum_{n=1}^\infty \|QI_n \sum_{i\in B_2} w_ig_i\|_{X_\alpha}e_{\min I_n}\Bigr\|_{X^2_\beta} <4\delta. \end{align*} Therefore \begin{align*} \|g\|_{G_{\alpha, \beta}} & \leqslant \ee+ \Bigl\|\sum_{n=1}^\infty \|QI_n \sum_{i\in B}w_ig_i\|_{X_\alpha} e_{\min I_n}\Bigr\|_{X^2_\beta} \\ &  \leqslant  \ee+ \Bigl\|\sum_{n=1}^\infty \|QI_n \sum_{i\in B_1}w_ig_i\|_{X_\alpha} e_{\min I_n}\Bigr\|_{X^2_\beta} + \Bigl\|\sum_{n=1}^\infty \|QI_n \sum_{i\in B_2}w_ig_i\|_{X_\alpha} e_{\min I_n}\Bigr\|_{X^2_\beta} \\ & < \ee+8\delta.\end{align*} Since $\delta>0$ was arbitrary, we are done.

$(ii)$ By \cite[Proposition $5$]{OSZ}, it is sufficient to show that for any $\ee>0$ and any weakly null tree $(h_E)_{E\in \mathcal{S}_{\beta+1}\setminus \{\varnothing\}}\subset B_{H_{\alpha, \beta}}$, there exist $E\in \mathcal{S}_{\beta+1}$ and $h\in \text{co}(h_F: \varnothing \prec F\preceq E)$ such that $\|h\|_{G_{\alpha,\beta}}<3\ee$.     Fix such a weakly null tree. By standard perturbation and pruning arguments, we may assume there is a block tree $(g_E)_{E\in \mathcal{S}_{\beta+1}}\subset B_{G_{\alpha, \beta}}$ such that $\|g_E-h_E\|_{G_{\alpha, \beta}}<\ee$ for all $E\in \mathcal{S}_{\beta+1}\setminus \{\varnothing\}$.  This yields that for each $E\in \mathcal{S}_{\beta+1}\setminus \{\varnothing\}$, $$\|Qg_E\|_{X_\alpha} =\|Qg_E-Qh_E\|_{X_\alpha} \leqslant \|g_E-h_E\|_{X_\alpha} \leqslant \ee.$$   By $(i)$, we may choose $\varnothing \neq E\in \mathcal{S}_{\beta+1}$ and non-negative numbers $(w_F)_{\varnothing\prec F\preceq E}$ such that $\sum_{\varnothing\prec F\preceq E} w_F=1$ and $\|\sum_{\varnothing\prec F\preceq E}w_F g_F\|_{G_{\alpha, \beta}}<2\ee$.   Then $$\Bigl\|\sum_{\varnothing \prec F\preceq E}w_Fh_F\Bigr\|_{G_{\alpha, \beta}} \leqslant \Bigl\|\sum_{\varnothing \prec F\preceq E}w_Fg_F\Bigr\|_{G_{\alpha, \beta}} + \sum_{\varnothing\prec F\preceq E}w_F\|h_F-g_F\|_{G_{\alpha, \beta}} <3\ee.$$   This finishes $(ii)$.

$(iii)$  Fix $0<\ee<1$.   We can recursively select positive integers $p_1<q_1<p_2<q_2<\ldots$ such that $\|f_{p_i}-f_{q_i}\|_{X_\alpha}<1-\ee$ for all $i\in\nn$. Since $Q$ is a quotient map, we can choose $(g_i)_{i=1}^\infty\subset B_{G_{\alpha, \beta}}$ such that $\|g_i\|_{G_{\alpha, \beta}}<1-\ee$ and $Qg_i=f_{q_i}-f_{p_i}$. Therefore  $h_i := e_{p_i}-e_{q_i}-g_i\in H_{\alpha, \beta}$.  By the triangle inequality, $\|h_i\|_{G_{\alpha, \beta}}\leqslant 3$.           Fix  $E\in \mathcal{S}_\beta$ and scalars $(b_i)_{i\in E}$ such that $\sum_{i\in E}|b_i|=1$ and note that  $(p_i:i\in E)\in \mathcal{S}_\beta$, since $\mathcal{S}_\beta$ is spreading. Let  \begin{displaymath}
   a_j = \left\{
     \begin{array}{lr}
       b_i & : j=p_i \text{\ for some\ }i\in E\\
       -b_i & : j=q_i \text{\ for some\ }i\in E\\
			0 & : \text{\ otherwise}.
     \end{array}
   \right.
\end{displaymath}  Let $I_i=\{p_i\}$ for $i\in E$ and note that \begin{align*} \Bigl\|\sum_{i\in E} b_i (e_{p_i}-e_{q_i}-g_i)\Bigl\|_{G_{\alpha, \beta}} & \geqslant \Bigl\|\sum_{i\in E} b_i (e_{p_i}-e_{q_i})\Bigl\|_{G_{\alpha, \beta}}-(1-\ee) \geqslant \Bigl\|\sum_{i \in E} \|\sum_{j\in I_i}a_jf_j\|_{X_\alpha} e_{p_i}\Bigr\|_{X_\beta^2} -(1-\ee) \\ & = \Bigl\|\sum_{j\in E} |b_i|e_{p_i}\Bigl\|_{X^2_\beta}-(1-\ee)= 1-(1-\ee)=\ee.\end{align*}   This yields that $(h_i)_{i=1}^\infty$ is an  $\ell_1^\beta$-spreading model.  By Remark \ref{remark}, $Sz(H_{\alpha, \beta})>\omega^\beta$, and $Sz(H_{\alpha, \beta})\geqslant \omega^{\beta+1}$ by \cite[Proposition $3.3$]{Lancien}.   Combining this estimate with  $(ii)$, we deduce that  $Sz(H_{\alpha, \beta})=\omega^{\beta+1}$.

\end{proof}

\begin{proof}[Proof of Theorems \ref{main1} and \ref{main2}] Fix $0<\xi<\omega_1$ such that $$\xi\notin \{\omega^\gamma: \gamma<\omega_1\}.$$  By standard properties of ordinals, there exist $\alpha, \beta<\xi$ such that $\alpha+\beta=\xi$. Let $G=G_{\alpha,\beta}$ and $H=H_{\alpha,\beta}$.

 Note that $G/H\approx X_\alpha$, so $Sz(G/H)=Sz(X_\alpha)=\omega^{\alpha+1}\leqslant \omega^\xi$. Here we are using the fact that  the Szlenk index is an isomorphic invariant (see \cite{Lancien}).  Since $G/H\approx X_\alpha$, $G/H$ admits no $\ell_1^{\alpha+1}$-spreading model. Since $\alpha+1\leqslant \xi$, as noted prior to Remark \ref{remark}, $G/H$ does not admit an $\ell_1^\xi$-spreading model.   It follows from Lemma \ref{finish} that $Sz(H)\leqslant \omega^{\beta+1}\leqslant \omega^\xi$ and that $Sz(H)$ admits no $\ell_1^{\beta+1}$-spreading model.  Since $\beta+1\leqslant \xi$, $H$ admits no $\ell_1^\xi$-spreading model.  By Remark \ref{remark}, we will be done once we show that $G$ admits an $\ell_1^\xi$-spreading model. Since $\alpha+\beta=\xi$, we will be done once we show that $G$ admits an $\ell_1^{\alpha+\beta}$-spreading model.

Fix $0<\ee<1$ and choose positive integers $m_1<m_2<\ldots$ such that for all $i\in\nn$, $\|f_{m_i}-e_i\|_{X_\alpha}<1-\ee$. Fix $E\in\mathcal{S}_\beta [\mathcal{S}_\alpha]$ and write $E=\cup_{n=1}^m E_n$ with $E_1<\ldots <E_m$, $\varnothing \neq E_n\in \mathcal{S}_\alpha$, $(\min E_n)_{n=1}^m \in \mathcal{S}_\beta$.  Let $F=\{m_i: i\in E\}$.    For each $1\leqslant n\leqslant m$, let $I_n$ be the smallest interval containing $\{m_i: i\in E_n\}$. Note that $(\min I_n)_{n=1}^m$ is a spread of $(\min E_n)_{n=1}^m$, since $\min I_n=m_{\min E_n}$, from which it follows that $(\min I_n)_{n=1}^m\in \mathcal{S}_\beta$.   Fix  scalars $(a_i)_{i\in E}$ such that $\sum_{i\in E}|a_i|=1$ and let $b_{m_i}=a_i$ for $i\in E$ and $b_i=0$ for $i\in\nn\setminus F$. Then  \begin{align*} \Bigl\|\sum_{i\in E} a_ie_{m_i}\Bigr\|_{G_{\alpha,\beta}}  & = \Bigl\|\sum_{i=1}^\infty b_ie_i\Bigr\|_{G_{\alpha,\beta}} \geqslant \Bigl\|\sum_{n=1}^m \|\sum_{i\in I_n} b_if_i\|_{X_\alpha} e_{\min I_n}\Bigr\|_{X_\beta^2} =\sum_{n=1}^m \Bigl\|\sum_{i\in I_n} b_if_i\Bigr\|_{X_\alpha} \\ & = \sum_{n=1}^m \Bigl\|\sum_{i\in E_n} a_if_{m_i}\Bigr\|_{X_\alpha}   \geqslant  \sum_{n=1}^m \Bigl\|\sum_{i\in E_n} a_ie_i\Bigr\|_{X_\alpha} -\sum_{i\in E} |a_i|\|f_{m_i}-e_i\|_{X_\alpha} \\ & = \sum_{n=1}^m \sum_{i\in E_n}|a_i| - \sum_{i\in E} |a_i|\|f_{m_i}-e_i\|_{X_\alpha}  > 1-(1-\ee)=\ee. \end{align*} Therefore by Remark \ref{remark}, $(e_{m_i})_{i=1}^\infty$ has a subsequence which is a an $\ell_1^{\alpha+\beta}$-spreading model.  By the criteria established at the end of the preceding paragraph, we are done.

\end{proof}

\end{document}